\documentclass[12pt]{amsart}
\usepackage{a4wide,amsmath,amssymb}

\newtheorem{theorem}{Theorem}   
\newtheorem{lemma}[theorem]{Lemma}   
\newtheorem{proposition}[theorem]{Proposition}

\let\pa\partial  
\let\na\nabla  
\let\eps\varepsilon  
\newcommand{\diver}{\operatorname{div}}    
\newcommand{\N}{{\mathbb N}}  
\newcommand{\R}{{\mathbb R}}   
\newcommand{\C}{{\mathbb C}} 

\newcommand{\V}{{\mathcal V}}

\begin{document}  

\title[Maxwell-Stefan systems]{Existence analysis of Maxwell-Stefan systems for
multicomponent mixtures}  
  
\author[A. J\"ungel]{Ansgar J\"ungel}
\address{Institute for Analysis and Scientific Computing, Vienna University of  
	Technology, Wiedner Hauptstra\ss e 8--10, 1040 Wien, Austria}
\email{juengel@tuwien.ac.at}
\author[I. V. Stelzer]{Ines Viktoria Stelzer}
\address{Institute for Analysis and Scientific Computing, Vienna University of  
	Technology, Wiedner Hauptstra\ss e 8--10, 1040 Wien, Austria}
\email{ines.stelzer@tuwien.ac.at} 

\thanks{The authors acknowledge partial support from   
 the Austrian Science Fund (FWF), grants P22108, P24304, and I395, and from
 the Austrian-French Project of the Austrian Exchange Service (\"OAD).
 The first author thanks Prof.\ W.~Dreyer for inspiring discussions on
 Maxwell-Stefan systems}  
 
\begin{abstract}
Maxwell-Stefan systems describing the dynamics of the molar concentrations
of a gas mixture with an arbitrary number of components are analyzed in a bounded
domain under isobaric, isothermal conditions. 
The systems consist of mass balance equations and equations for the chemical 
potentials, depending on the relative velocities, supplemented with
initial and homogeneous Neumann boundary conditions.
Global-in-time existence of bounded weak solutions to the quasilinear parabolic system
and their exponential decay to the homogeneous steady state are proved.
The mathematical difficulties are due to the singular Maxwell-Stefan diffusion matrix,
the cross-diffusion coupling, and the lack of standard maximum principles.
Key ideas of the proofs are the Perron-Frobenius theory for quasi-positive matrices, 
entropy-dissipation methods, 
and a new entropy variable formulation allowing for the proof of 
nonnegative lower and upper bounds for the concentrations.
\end{abstract}  
 
% \paragraph{Keywords:}  
 
% \paragraph{AMS classification:}  
 
\maketitle  
 
%%%%%%%%%%%%%%%%%%%%%%%%%%%%%%%%%%%%%%%%%%%%%%%%%%%%%%%%%%%%%%%%%%%%%%%%%%%%%%%%  

\section{Introduction}  

The Maxwell-Stefan equations describe the diffusive transport of multicomponent mixtures \cite{Max66,Ste71}. 
%They have been developed by Maxwell for dilute gases \cite{Max66} and by Stefan
%for fluids \cite{Ste71}. 
Applications include various fields like sedimentation, dialysis, electrolysis,
ion exchange, ultrafiltration, and respiratory airways \cite{BGG10,WeKr00}.
The model bases upon inter-species force balances,
relating the velocities of the species of the mixture. 
It is well-known that the usual Fickian diffusion model, 
which states that the flux of a chemical substance is 
proportional to its concentration gradient, is not able to describe, e.g.,
uphill or osmotic diffusion phenomena in multicomponent mixtures,
as demonstrated experimentally by Duncan and Toor \cite{DuTo62}.
These phenomena can be modeled by using the theory of non-equilibrium thermodynamics,
in which the fluxes are assumed to be linear 
combinations of the thermodynamic forces \cite[Chap.~4]{DeMa62}. 
However, this model requires
the knowledge of all binary diffusion coefficients, which are not always easy to
determine, and the positive semi-definiteness of the diffusion matrix. 
The advantage of the Maxwell-Stefan approach is that
it is capable to describe uphill diffusion effects without assuming particular
properties on the diffusivities (besides symmetry).

We consider an ideal gaseous mixture consisting of $N+1$ components with molar 
concentrations $c_i(x,t)$ for $i=1,\ldots,N+1$ (see Appendix \ref{sec.deriv}).
Since we concentrate our study on cross-diffusion effects, we suppose isothermal
and isobaric conditions. 
Then the total molar concentration $\sum_{i=1}^{N+1}c_i$ is constant and we set
this constant equal to one. 
More general situations are investigated, e.g., in \cite{Gio99}.
The dynamics of the mixture is given by the mass balance equations
\begin{equation}\label{eq1}
  \pa_t c_i + \diver J_i = r_i(c) \quad\mbox{in }\Omega,\ t>0,\ i=1,\ldots,N+1,
\end{equation}
where $J_i=c_i u_i$ denotes the molar flux, $u_i$ the mean velocity,
$r_i$ the net production rate of the $i$-th component, 
$c=(c_1,\ldots,c_{N+1})^\top$, and $\Omega\subset\R^d$ ($d\ge 1$) is a bounded domain.
We have assumed above that the averaged mean velocity 
vanishes, $\sum_{i=1}^{N+1}c_i u_i=0$.
Then the conservation of the total mass implies
that the total production rate vanishes, $\sum_{i=1}^{N+1}r_i(c)=0$.
The fluxes are related to the concentration gradients by
\begin{equation}\label{eq2}
  \na c_i = -\sum_{j=1,\,j\neq i}^{N+1}\frac{c_j J_i-c_i J_j}{D_{ij}}, \quad
  i=1,\ldots,N+1,
\end{equation}
where $D_{ij}>0$ for $i\neq j$ are some diffusion coefficients.
The derivation of these relations is sketched in Appendix \ref{sec.deriv}.

The aim of this paper is to prove the global-in-time existence of weak solutions
to system \eqref{eq1}-\eqref{eq2} for constant coefficients
$D_{ij}>0$, supplemented with the boundary and initial conditions
\begin{equation}\label{bic}
  \na c_i\cdot\nu = 0\quad\mbox{on }\pa\Omega,\ t>0, \quad
  c_i(\cdot,0)=c_i^0\quad\mbox{in }\Omega, \quad i=1,\ldots,N+1,
\end{equation}
where $\nu$ is the exterior unit normal vector on $\pa\Omega$.
There are several difficulties to overcome in the analysis of the 
Maxwell-Stefan system. 

{\em First}, the 
molar fluxes are not defined {\em a priori}
as a linear combination of the concentration
gradients, which makes it necessary to invert the flux-gradient relations \eqref{eq2}.
As the Maxwell-Stefan equations are linearly dependent, 
we need to invert the system on a subspace, yielding the
diffusion matrix $\tilde A^{-1}$.
In the engineering literature, this inversion is often done in an approximate way. 
For instance, a numerical solution procedure for $N=3$ was developed in \cite{APP03}
and the special case $D_{ij}=1/(f_if_j)$ for some constants $f_i>0$
was investigated in \cite{BKR68}.

{\em Second}, equations \eqref{eq1}-\eqref{eq2} are coupled, which translates
into the fact that $\tilde A^{-1}$ is generally a full matrix with nonlinear
solution-dependent coefficients. Thus, standard tools 
like the maximum principle or regularity theory are not available. 
In particular, it is not clear how to prove nonnegative lower
and upper bounds for the concentrations. Moreover, it is not clear whether
$\tilde A^{-1}$ is positive semi-definite or not, such that even the
proof of local-in-time existence of solutions is nontrivial.

{\em Third}, it is not standard to find suitable a priori estimates which
allow us to conclude the global-in-time existence of solutions.

In view of these difficulties, it is not surprising that
there are only very few analytical results
in the mathematical literature for Maxwell-Stefan systems. Under some
general assumptions on the nonlinearities, Giovangigli proved that there exists
a unique global solution to the whole-space Maxwell-Stefan system if
the initial datum is sufficiently close to the equilibrium state 
\cite[Theorem 9.4.1]{Gio99}. Bothe \cite{Bot11} showed the existence of a unique local
solution for general initial data. Boudin et al.\ \cite{BGS12} considered
a ternary system ($N=2$) and assumed that two diffusivities are equal. In this
situation, the Maxwell-Stefan system reduces to a heat equation for the
first component and a drift-diffusion-type equation for the second species.
Boudin et al.\ \cite{BGS12}
proved the existence of a unique global solution and investigated its
long-time decay to the stationary state.
Up to now, there does not exist a global existence theory for \eqref{eq1}-\eqref{eq2} 
for general initial data. We provide such a result in this work.

After inverting the flux-gradient relations \eqref{eq2} on a suitable subspace, 
the Maxwell-Stefan equations
become a parabolic cross-diffusion system. Amann derived in \cite{Ama89}
sufficient conditions for the solutions to such systems to exist globally
in time. The question if a given local solution exists globally is reduced
to the problem of finding a priori estimates in suitable Sobolev spaces. 
Another approach was developed in \cite{ChJu04,DGJ97,GVJ03,HiJu11,JuSt12} 
for systems arising in granular material modeling, population dynamics,
cell biology, and thermodynamics to treat cross-diffusion systems whose
diffusion matrix may be neither symmetric nor positive semi-definite.
The idea is to exploit the entropy structure of the model by introducing 
so-called entropy variables. In these variables, the new diffusion matrix becomes
positive semi-definite and an entropy-dissipation relation can be derived.
However, in all of the mentioned papers (except \cite{DGJ97}) systems
with two equations only have been considered.

In this paper, we combine and extend the entropy-dissipation technique of
\cite{ChJu04,DGJ97,GVJ03,HiJu11,JuSt12} as well as ideas of Bothe \cite{Bot11}
to overcome the above mathematical difficulties. We are able to prove the
global-in-time existence of weak solutions to \eqref{eq1}-\eqref{bic} for 
arbitrary diffusion matrices and general initial data. This result is obtained
under the following assumptions:
\begin{itemize}
\item Domain: $\Omega\subset\R^d$ ($d\le 3$) is a bounded domain with 
$\pa\Omega\in C^{1,1}$.
\item Initial data: $c_1^0,\ldots,c_{N}^0$ ($N\ge 2$)
are nonnegative measurable functions, $c_{N+1}^0=1-\sum_{i=1}^N c_i^0$, and
\begin{equation}\label{ass.init}
  \sum_{i=1}^N c_i^0\le 1.
\end{equation}
\item Diffusion matrix: $(D_{ij})\in\R^{(N+1)\times(N+1)}$ is a symmetric
matrix with elements $D_{ij}>0$ for $i\neq j$.
\item Production rates: The functions $r_i\in C^0([0,1]^{N+1};\R)$, $i=1,\ldots,N+1$,
satisfy
\begin{equation}\label{ass.r}
  \sum_{i=1}^{N+1}r_i(c)=0,\quad \sum_{i=1}^{N+1}r_i(c)\log c_i\le 0
  \quad\mbox{for all }0 < c_1,\ldots,c_{N+1}\le 1.
\end{equation}
\end{itemize}

We stress the fact that, although the diffusion coefficients $D_{ij}$ are 
constant, the diffusion matrix of the inverted Maxwell-Stefan system
depends on the molar concentrations in a nonlinear way (see below) and we need to 
deal with a fully coupled nonlinear parabolic system. Our proof also works
for diffusion coefficients depending on the concentrations $c_i$ 
if the coefficients $d_{ij}$ are bounded from above and below.

The regularity on the boundary $\pa\Omega$ is needed for the a priori estimate
$\|w\|_{H^2(\Omega)}\le C\|f\|_{L^2(\Omega)}$
of the elliptic problem $-\Delta w+w=f$ in $\Omega$, $\na w\cdot\nu=0$ on
$\pa\Omega$.

The inequality imposed on the production rates is needed to prove that the
entropy is nonincreasing in time. It is satisfied if, for instance,
$N=4$ and $r_1=r_3=c_2c_4-c_1c_3$, $r_2=r_4=c_1c_3-c_2c_4$
\cite{DFPV07}. For the existence result, the inequality can be weakened by
\begin{equation}\label{ass.r2}
  \sum_{i=1}^{N+1}r_i(c)\log c_i\le C_r
  \quad\mbox{for all }0 < c_1,\ldots,c_{N+1}\le 1,
\end{equation}
where $C_r>0$ is some constant independent of $c_i$ (see \eqref{gen.r}). 
This condition is satisfied, 
for instance, for the tumor-growth model in \cite{JuSt12}.

Our first main result is the global existence of solutions to \eqref{eq1}-\eqref{bic}.

\begin{theorem}[Global existence of solutions]\label{thm.ex}
Let the above assumptions hold. Then there exists a weak solution $(c_1,\ldots,c_{N+1})$
to \eqref{eq1}-\eqref{bic} satisfying
\begin{align*}
  & c_i\in L^2_{\rm loc}(0,\infty;H^{1}(\Omega)), \quad 
  \pa_t c_i\in L^2_{\rm loc}(0,\infty;\V'), \\
  & 0\le c_i\le 1, \quad i=1,\ldots, N, \quad c_{N+1}=1-\sum_{i=1}^N c_i\ge 0
  \quad\mbox{in }\Omega,\ t>0,
\end{align*}
where $\V'$ is the dual space of $\V=\{u\in H^2(\Omega):\na u\cdot\nu = 0$
on $\pa\Omega\}$. 
\end{theorem}

To be precise, the existence theorem for \eqref{eq1}-\eqref{bic}
has to be understood as an existence result for a system in $N$ components
(see below), which is equivalent to \eqref{eq1}-\eqref{bic} as long as $c_i>0$ 
for all $i=1,\ldots,N$ and $\sum_{i=1}^N c_i<1$ are satisfied.

We explain the key ideas of the proof. For this, we write \eqref{eq2}
more compactly as
\begin{equation}\label{eq2a}
  \na c = A(c)J,
\end{equation}
where $A(c)\in\R^{(N+1)\times(N+1)}$ and $\na c=(\pa c_i/\pa x_j)_{ij}$, 
$J=(J_1,\ldots,J_{N+1})^\top\in\R^{(N+1)\times d}$. Using the Perron-Frobenius
theory for quasi-positive matrices, Bothe \cite{Bot11} characterized the
spectrum of $A(c)$ in case that $c_i>0$ for $i=1,\ldots,N$ and $\sum_{i=1}^N c_i<1$. 
Under these conditions, $A(c)$ can be inverted on its image. Then,
denoting its inverse by $\tilde A(c)^{-1}$, 
\eqref{eq1}-\eqref{eq2} can be formulated as
\begin{equation}\label{eq.A}
  \pa_t c - \diver(-\tilde A(c)^{-1}\na c) = r(c)\quad\mbox{in }\Omega,\ t>0,
\end{equation}
where $r(c)=(r_1(c),\ldots,r_{N+1}(c))^\top$.

It turns out that it is more convenient to eliminate the last equation for
$c_{N+1}$, which is determined by $c_{N+1}=1-\sum_{i=1}^N c_i$, and to work
only with the system in $N$ components.
We set $c'=(c_1,\ldots,c_{N})^\top$, $J'=(J_1,\ldots,J_{N})^\top$, and $r'(c)=(r_1(c),\ldots,r_N(c))^\top$. 
Using the facts that $c_{N+1}=1-\sum_{i=1}^N c_i$ and $\sum_{i=1}^N J_i=-J_{N+1}$,
system \eqref{eq2a} can be written as $\na c'=-A_0(c') J'$.
The matrix $A_0(c')$ defined in Section \ref{sec.matrix} is generally not symmetric
and it is not clear if it is positive definite.
If $A_0(c')$ is invertible (and we prove in Section \ref{sec.matrix} that this
is the case),  we can write system \eqref{eq1}-\eqref{eq2} as
\begin{equation}\label{eq.A0}
  \pa_t c' - \diver(A_0(c')^{-1}\na c') = r'(c) \quad\mbox{in }\Omega,\ t>0.
\end{equation}
Still, $A_0(c')^{-1}$ may be not positive (semi-) definite.

Our main idea to handle \eqref{eq.A0} is to exploit its entropy structure. 
We associate to this system the entropy density
\begin{equation}\label{def.h}
  h(c') = \sum_{i=1}^{N}c_i(\log c_i-1) + c_{N+1}(\log c_{N+1}-1), 
  \quad c_1,\ldots,c_{N}\ge 0,\ \sum_{i=1}^N c_i\le 1,
\end{equation}
where $c_{N+1}=1-\sum_{i=1}^N c_i$ is interpreted as a function of the 
other concentrations. Furthermore, we define the entropy variables
\begin{equation}\label{def.w}
  w_i = \frac{\pa h}{\pa c_i} = \log\frac{c_i}{c_{N+1}}, \quad i=1,\ldots,N,
\end{equation}
and we denote by $H(c')=\na^2 h(c')$ the Hessian of $h$ with respect 
to $c'$. Then \eqref{eq.A0} becomes
\begin{equation}\label{eq.B}
  \pa_t c' - \diver(B(w)\na w) = r'(c)\quad\mbox{in }\Omega,\ t>0,
\end{equation}
where $w=(w_1,\ldots,w_N)^\top$ and $B(w)=A_0(c)^{-1}H(c)^{-1}$ is 
symmetric and positive definite (see Lemma \ref{lem.b}).
The advantage of the formulation in terms of the entropy variables
is not only that the diffusion matrix $B(w)$ is positive definite (which allows us
to apply the Lax-Milgram lemma to a linearized version of \eqref{eq.B})
but it yields also positive lower and upper bounds for the concentrations. 
Indeed, inverting \eqref{def.w}, we find that
\begin{equation}\label{inv.w}
  c_i = \frac{e^{w_i}}{1+e^{w_1}+\cdots+e^{w_N}}, \quad i=1,\ldots,N.
\end{equation}
Therefore, if the functions $w_i$ are bounded, the concentrations
$c_i$ are positive and $\sum_{i=1}^N c_i < 1$
which implies that $c_{N+1}=1-\sum_{i=1}^N c_i > 0$. This observation is
a key novelty of the paper.

Formulation \eqref{eq.A} is needed to derive a priori estimates which are an
important ingredient for the global existence proof.
Differentiating the entropy ${\mathcal H}[c]=\int_\Omega h(c)dx$ 
(now $h(c)$ is interpreted as a function of all $c_1,\ldots,c_{N+1}$)
formally with respect to time, a computation (made rigorous in Lemma \ref{lem.ent}) 
shows the entropy-dissipation inequality
\begin{equation}\label{ent.ineq}
  \frac{d{\mathcal H}}{dt} + K\int_\Omega\sum_{i=1}^{N+1}|\na\sqrt{c_i}|^2 dx \le 0,
\end{equation}
where $K>0$ is a constant which depends only on $(D_{ij})$. This estimate
yields $H^1$ bounds for $\sqrt{c_i}$.

The existence proof is based on the construction of a problem which approximates
\eqref{eq.B}. We replace the time derivative by an implicit Euler discretization
with time step $\tau>0$ and we add the fourth-order operator 
$\eps(\Delta^2 w+w)$, which guarantees the uniform coercivity of the elliptic
system in $\V$ with respect to $w$. The existence of approximating weak 
solutions is shown by means of the Leray-Schauder fixed-point theorem. 
The discrete analogon of the above entropy-dissipation estimate implies a priori
bounds uniform in the approximation parameters $\tau$ and $\eps$, which allows us
to pass to the limit $(\tau,\eps)\to 0$. In particular, the entropy inequality
provides global solutions.

System \eqref{eq1}-\eqref{bic} with vanishing production rates, 
$r=(r_1,\ldots,r_{N+1})^\top=0$, admits the
homogeneous steady state $(\bar c_1^0,\ldots,\bar c_{N+1}^0)$, where
$\bar c_i^0=\mbox{meas}(\Omega)^{-1}\int_\Omega c_i^0 dx$. We are able to prove
that the solution, constructed in Theorem \ref{thm.ex}, converges exponentially
fast to this stationary state. For this, we introduce the relative entropy
$$
  {\mathcal H}^*[c] = \sum_{i=1}^{N+1}\int_\Omega c_i\log\frac{c_i}{\bar c_i^0}dx.
%  \quad\mbox{where}\quad \bar c_i = \frac{1}{\mbox{meas}(\Omega)}\int_\Omega c_i dx.
$$

\begin{theorem}[Exponential decay]\label{thm.long}
Let the assumptions of Theorem \ref{thm.ex} hold.
We suppose that $r=0$ and $\min_{i=1,\ldots,N+1}\|c_i^0\|_{L^1(\Omega)}>0$.
Let $(c_1,\ldots,c_{N+1})$ be the weak solution constructed in Theorem \ref{thm.ex}
and define $c^0=(c_1^0,\ldots,c_{N+1}^0)^\top$.
Then there exist constants $C>0$, depending only on $\Omega$, and
$\lambda>0$, depending only on $\Omega$ and $(D_{ij})$, such that
$$
  \|c_i(\cdot,t)-\bar c_i^0\|_{L^1(\Omega)} 
  \le Ce^{-\lambda t}\sqrt{{\mathcal H}^*[c^0]}, \quad i=1,\ldots,N+1.
$$
\end{theorem}

The proof of this result is based on the entropy-dissipation inequality
\eqref{ent.ineq} and the logarithmic Sobolev inequality \cite{Gro75}, which links the
entropy dissipation to the relative entropy, as well as on the Csisz\'ar-Kullback
inequality \cite{UAMT00}, which bounds the squared $L^1$ norm in terms of
the relative entropy. The difficulty of the proof is that the approximate 
solutions do not conserve the $L^1$-norm because of the presence of the
regularizing $\eps$-terms, and we need to derive appropriate bounds. 

The paper is organized as follows. In Section \ref{sec.matrix}, we prove some
properties of the diffusion matrices $A(c)$ and $A_0(c)$
and we show how system \eqref{eq1}-\eqref{eq2} of $N+1$ equations 
can be reduced to a system of $N$ equations. Based on these properties, 
Theorems \ref{thm.ex} and \ref{thm.long} are proved in Sections \ref{sec.ex}
and \ref{sec.long}, respectively. For the convenience of the reader, 
the derivation of the Maxwell-Stefan relations
\eqref{eq2} is sketched in Appendix \ref{sec.deriv} and some definitions and 
results from matrix theory needed in Section \ref{sec.matrix}
are summarized in Appendix \ref{sec.app}.

%%%%%%%%%%%%%%%%%%%%%%%%%%%%%%%%%%%%%%%%%%%%%%%%%%%%%%%%%%%%%%%%%%%%%%%%%%%%%%%%  

\section{Properties of the diffusion matrices}\label{sec.matrix}

Let the Maxwell-Stefan diffusion matrix $(D_{ij})\in\R^{(N+1)\times(N+1)}$ 
($N\ge 2$) be symmetric with $D_{ij}>0$ for $i\neq j$ and $D_{ii}=0$ for all $i$
and set $d_{ij}=1/D_{ij}$ for $i\ne j$.
Let $c=(c_i)\in\R^{N+1}$ be a strictly positive vector satisfying 
$\sum_{i=1}^{N+1}c_i=1$. 
We refer to Appendix \ref{sec.app} for the definitions and 
results from matrix analysis used in this section. 
According to \eqref{eq2} and \eqref{eq2a},
the matrix $A=A(c)=(a_{ij})\in\R^{(N+1)\times(N+1)}$ is given by
$$
  a_{ij} = d_{ij}c_i\quad\mbox{for }i,j=1,\ldots,N+1,\ i\neq j, \quad
  a_{ii} = -\sum_{j=1,\,j\neq i}^{N+1} d_{ij}c_j \quad\mbox{for }i=1,\ldots,N+1.
$$
In \cite[Section 7.7.1]{Gio99}, the matrix with elements $-a_{ij}c_j$ is
analyzed and it is shown that it is symmetric, positive semi-definite,
irreducible, and a singular M-matrix as well as that a generalized inverse can be
defined. Our approach is to apply the Perron-Frobenius theory to $A$, 
following \cite{Bot11}.

\begin{lemma}[Properties of $A$]\label{lem.a}
Let $\delta=\min_{i,j=1,\ldots,N+1,\,i\neq j}d_{ij}>0$ and 
$\Delta=2\sum_{i,j=1,\,i\neq j}^{N+1}d_{ij}$. Then the spectrum $\sigma(-A)$ of
$-A$ satisfies
$$
  \sigma(-A) \subset \{0\}\cup [\delta,\Delta).
$$
\end{lemma}

The inclusion $\sigma(-A)\subset \{0\}\cup[\delta,\infty)$ is shown in
\cite[Section 5]{Bot11}. For the convenience of the reader and since
some less known results from matrix analysis are needed, we present a full proof.

\begin{proof}
%We assume that $0<c_i<1$ for $i=1,\ldots,N*1$.
The matrix $A$ is quasi-positive and irreducible. Therefore, by Theorem \ref{thm.PF}
of Perron-Frobenius (see Appendix \ref{sec.app}), 
the spectral bound of $A$, $s(A)=\max\{\Re(\lambda):\lambda\in
\sigma(A)\}$, is a simple eigenvalue of $A$ associated 
with a strictly positive eigenvector
and $s(A)>\Re(\lambda)$ for all $\lambda\in\sigma(A)$, $\lambda\neq s(A)$.
Here, $\Re(z)$ denotes the real part of the complex number $z$. Thus,
$$
  \sigma(A)\subset \{s(A)\}\cup \{z\in\C:\Re(z)<s(A)\}.
$$
An elementary computation shows that $c$ is a (strictly) positive eigenvector
to the eigenvalue $\lambda=0$ of $A$. According to the Perron-Frobenius theory,
only the eigenvector to $s(A)$ is positive. This implies that $s(A)=0$ and
$$
  \sigma(A)\subset \{0\}\cup \{z\in\C:\Re(z)<0\}.
$$

We can describe the spectrum of $\sigma(A)$ in more detail.  
Let $C^{1/2}=\mbox{diag}(\sqrt{c_1},\ldots,\sqrt{c_{N+1}})$ be a diagonal matrix
in $\R^{(N+1)\times(N+1)}$ with inverse $C^{-1/2}$. Then we can introduce the
symmetric matrix $A_S = C^{-1/2}AC^{1/2}$ whose elements are given by
$$
  a^S_{ij} = \left\{\begin{array}{cl}
  a_{ii} &\quad\mbox{if }i=1,\ldots,N+1, \\
  d_{ij}\sqrt{c_i c_j} &\quad\mbox{if }i,j=1,\ldots,N+1,\ i\neq j.
  \end{array}\right.
$$
The matrix $A_S$ is real and symmetric since $d_{ij}=d_{ji}$
and therefore, it has only real eigenvalues.
Since $A$ and $A_S$ are similar, their spectra coincide:
$$
  \sigma(A_S) = \sigma(A) \subset \{0\}\cup\{z\in\R:z<0\} = (-\infty,0].
$$

Now, consider the matrix $A_S(\alpha)=A_S-\alpha\sqrt{c}\otimes\sqrt{c}$,
where $0<\alpha<\delta$ and $\sqrt{c}=(\sqrt{c_1},\ldots,\sqrt{c_{N+1}})^\top$.
Then $A_S(\alpha)$ is quasi-positive and irreducible (since $\alpha<\delta\le d_{ij}$).
Using $\sum_{i=1}^{N+1}c_i=1$, a computation shows that $-\alpha$ is an eigenvalue 
of $A_S(\alpha)$ associated to the strictly positive eigenvector $\sqrt{c}$. 
By Theorem \ref{thm.PF} of Perron-Frobenius, the spectral bound
of $A_S(\alpha)$ equals $-\alpha$ and
$$
  \sigma(A_S(\alpha))\subset (-\infty,-\alpha].
$$
Since $A_S(\alpha)$ and $\alpha\sqrt{c}\otimes\sqrt{c}$ are symmetric,
we can apply Theorem \ref{thm.weyl} of Weyl:
$$
  \lambda_i(A_S) = \lambda_i\big(\alpha\sqrt{c}\otimes\sqrt{c} + A_S(\alpha)\big)
  \le \lambda_i\big(\alpha\sqrt{c}\otimes\sqrt{c}\big) + \lambda_{N+1}(A_S(\alpha)), 
$$
where $i=1,\ldots,N+1$ and
the eigenvalues $\lambda_i(\cdot)$ are arranged in increasing order.
Because of $\lambda_{N+1}(A_S(\alpha))=-\alpha$ and
$\lambda_i(\alpha\sqrt{c}\otimes\sqrt{c})=0$ for $i=1,\ldots,N$,
$\lambda_{N+1}(\alpha\sqrt{c}\otimes\sqrt{c})$ $=\alpha$
(see Proposition~\ref{prop.xy}), we find that $\lambda_i(A_S)\le -\alpha$ for $i=1,\ldots,N$ 
and $\lambda_{N+1}(A_S)\le 0$. Thus, for all $\alpha<\delta$,
$$
  \sigma(A) = \sigma(A_S) \subset \{0\}\cup(-\infty,-\alpha],
$$
implying that $\sigma(-A)\subset\{0\}\cup[\delta,\infty)$. 

It remains to prove the upper bound of the spectrum. Denoting by $\|\cdot\|_F$
the Frobenius norm, we find for the spectral radius of $-A$ that
\begin{align*}
  r(-A) &\le \|-A\|_F = \left(\sum_{i,j=1}^{N+1}a_{ij}^2\right)^{1/2}
  = \left(\sum_{i=1}^{N+1}\left(\sum_{j=1,\,j\neq i}^{N+1}d_{ij} c_j\right)^2
  + \sum_{i,j=1,\,j\neq i}^{N+1}(d_{ij}c_i)^2\right)^{1/2} \\
  &< 2\sum_{i,j=1,\,j\neq i}^{N+1}d_{ij} = \Delta,
\end{align*}
since $0<c_i<1$, finishing the proof.
\end{proof}

\begin{lemma}[Properties of restrictions of $A$ and $A_S$]\label{lem.tildea}
Let $\tilde A=A|_{\text{\rm im}(A)}$ and $\tilde A_S=A_S|_{\text{\rm im}(A_S)}$.
Then $\tilde A$ and $\tilde A_S$ are invertible on the images $\text{\rm im}(A)$ and
$\text{\rm im}(A_S)$, respectively, and
\begin{equation}\label{prop.b}
  \sigma(-\tilde A),\,\sigma(-\tilde A_S)\subset [\delta,\Delta), \quad
  \sigma((-\tilde A_S)^{-1})\subset (1/\Delta,1/\delta].
\end{equation}
\end{lemma}

\begin{proof}
Direct inspection shows that $\mbox{ker}(A)=\mbox{span}\{c\}$,
$\mbox{im}(A)=\{\mathbf{1}\}^\perp$, where $\mathbf{1}=(1,\ldots,$ $1)^\top
\in\R^{N+1}$, and $\mbox{ker}(A_S)=\mbox{span}\{\sqrt{c}\}$. 
By the symmetry of $A_S$, it follows that
\begin{equation}\label{sum.as}
  \R^{N+1}=\mbox{ker}(A_S)^\perp\oplus\mbox{ker}(A_S)
  =\mbox{im}(A_S^\top)\oplus\mbox{ker}(A_S)=\mbox{im}(A_S)\oplus\mbox{ker}(A_S).
\end{equation}
Furthermore, using Theorem \ref{thm.semi}, since $\lambda=0$ is a semisimple 
eigenvalue of $A$,
\begin{equation}\label{sum.a}
  \R^{N+1}=\mbox{im}(A)\oplus\mbox{ker}(A).
\end{equation}

We observe that both $\tilde A$ and $\tilde A_S$ are endomorphisms. 
Clearly, $\sigma(\tilde A)\subset\sigma(A)$ and $\sigma(\tilde A_S)\subset
\sigma(A_S)$. We claim that $0$ is not contained in $\sigma(\tilde A)$ or
$\sigma(\tilde A_S)$.
Indeed, otherwise there exists $x\in\text{im}(A)$
(or $x\in\text{im}(A_S)$), $x\neq 0$, such that $\tilde Ax=0$
(or $\tilde A_S x=0$). But this implies that $x\in\mbox{ker}(A)$ 
(or $x\in\mbox{ker}(A_S)$) and because of \eqref{sum.a} (or \eqref{sum.as}),
it follows that $x=0$, contradiction. Hence, $\tilde A$ and $\tilde A_S$ are
invertible on their respective domain, and \eqref{prop.b} follows.
\end{proof}

The above lemma shows that the flux-gradient relation 
\eqref{eq2a} can be inverted since $\sum_{i=1}^{N+1}J_i$ $=0$ implies that 
each column of $J$ is an element of 
$\{\mathbf{1}\}^\perp=\mbox{im}(A)$. In fact, we can write \eqref{eq2a}
as $\na c=\tilde A J$ and hence, $J=\tilde A^{-1}\na c$. Therefore, we can formulate
\eqref{eq1} and \eqref{eq2} as
\begin{equation}\label{eq.c}
  \pa_t c - \diver(-\tilde A^{-1}\na c) = r(c)\quad\mbox{in }\Omega,\ t>0.
\end{equation}

The next step is to reduce the Maxwell-Stefan system of $N+1$ components to
a system of $N$ components only. Still, we assume that $c_i>0$ for all $i$
and $\sum_{i=1}^{N+1}c_i=1$. We define the matrices
$$
  X = I_{N+1} - \begin{pmatrix} 0 \\ \vdots \\ 0 \\ 1 \end{pmatrix}
  \otimes \begin{pmatrix} 1 \\ \vdots \\ 1 \\ 0 \end{pmatrix}, \quad
  X^{-1} = I_{N+1} + \begin{pmatrix} 0 \\ \vdots \\ 0 \\ 1 \end{pmatrix}
  \otimes \begin{pmatrix} 1 \\ \vdots \\ 1 \\ 0 \end{pmatrix}
  \in\R^{(N+1)\times(N+1)},
$$
where $X^{-1}$ is the inverse of $X$ and $I_{N+1}$ is the unit matrix of
$\R^{(N+1)\times(N+1)}$.
%  X = I - (0,\ldots,0,1)^\top\otimes(1,\ldots,1,0)^\top\in\R^{(N+1)\times(N+1)},
%where $I$ is the unit matrix, with inverse
%$X^{-1}=I+(0,\ldots,0,1)^\top\otimes(1,\ldots,1,0)^\top$.
%$$
%  x_{ij} = \left\{\begin{array}{cl}
%  \delta_{ij} &\quad\mbox{if }i=1,\ldots,N,\ j=1,\ldots,N+1, \\
%  -1 &\quad\mbox{if }i=N+1,\ j=1,\ldots,N, \\
%  1 &\quad\mbox{if }i=j=N+1.
%  \end{array}\right.
%$$
A computation shows that 
$$
  X^{-1}AX = \begin{pmatrix} -A_0 & b \\ 0 & 0 \end{pmatrix},
$$
where the $(N\times N)$-matrix $A_0=(a_{ij}^0)$ is defined by
\begin{equation}\label{def.a0}
  a^0_{ij} = \left\{\begin{array}{cl}
  -(d_{ij}-d_{i,N+1})c_i &\quad\mbox{if }i\neq j,\ i,j=1,\ldots,N, \\
  \sum_{j=1,\,j\neq i}^N (d_{ij}-d_{i,N+1})c_j + d_{i,N+1}
  &\quad\mbox{if }i=j=1,\ldots,N,
  \end{array}\right.
\end{equation}
and the vector $b=(b_i)$ is given by $b_i=d_{i,N+1}c_i$, $i=1,\ldots,N$.
In Lemma \ref{lem.a0} below we show that $A_0$ is invertible.
Then, writing $c'=(c_1,\ldots,c_N)^\top$ and $J'=(J_1,\ldots,J_N)^\top$,
$$
  \begin{pmatrix} \na c' \\ 0 \end{pmatrix}
  = X^{-1}\na c = (X^{-1}AX)X^{-1} J 
  = \begin{pmatrix} -A_0 J' \\ 0 \end{pmatrix}.
$$
Thus, applying $X^{-1}$ to \eqref{eq.c}, since $X^{-1}\pa_t c = (\pa_t c',0)^\top$ 
and $X^{-1}r(c)=(r'(c),0)^\top$,
\begin{equation}\label{eq.cp}
  \pa_t c' - \diver(A_0^{-1}\na c') = r'\quad\mbox{in }\Omega,\ t>0.
\end{equation}
We note that every solution $c'$ to this problem defines a solution to
\eqref{eq.c} and hence to \eqref{eq1}-\eqref{eq2} by multiplying \eqref{eq.cp}
(augmented via $(c',0)^\top$) by $X$ and setting $c_{N+1}=1-\sum_{i=1}^N c_i$.

\begin{lemma}[Properties of $A_0$]\label{lem.a0}
The matrix $A_0\in\R^{N\times N}$, defined in \eqref{def.a0}, is invertible
with spectrum
$$
  \sigma(A_0)\subset [\delta,\Delta).
$$
Furthermore, the elements of its inverse $A_0^{-1}$ are uniformly bounded in 
$c_1,\ldots,c_N\in[0,1]$.
\end{lemma}

\begin{proof}
Since the blockwise upper triangular matrix $-X^{-1}AX$ is similar to $-A$,
their spectra coincide and
\begin{equation}\label{sigma.a0}
  \sigma(A_0)\cup\{0\} = \sigma(-X^{-1}AX) = \sigma(-A) 
  \subset \{0\}\cup[\delta,\Delta).
\end{equation}
Observing that $0$ is a simple eigenvalue of $-A$, it follows that
$\sigma(A_0)\subset[\delta,\Delta)$ and hence, $A_0$ is invertible.

It remains to show the uniform bound for the elements $\alpha_{ij}$ of $A_0^{-1}$.
By Cramer's rule, $A_0^{-1}=\mbox{adj}(A_0)/\det A_0$, where $\mbox{adj}(A_0)$
is the adjugate of $A_0$. The definition of $A_0$ in \eqref{def.a0} implies that
$$
  |a^0_{ij}| \le \sum_{k=1,\,k\neq i}^N|d_{ik}-d_{i,N+1}| + |d_{i,N+1}| = K_i\le K, 
  \quad i,j=1,\ldots,N,
$$
where $K=\max_{i=1,\ldots,N}K_i$.
Therefore, the elements of $\mbox{adj}(A_0)$ are not larger than $(N-1)!K^{N-1}$.
By \eqref{sigma.a0}, the eigenvalues of $A_0$ are bounded from below by $\delta$.
Consequently, since the determinant of a matrix equals the product of its eigenvalues, 
$\det(A_0)\ge\delta^N$. This shows that 
$|\alpha_{ij}|\le (N-1)!K^{N-1}\delta^{-N}$ for all $i$, $j$.
\end{proof}

Consider the Hessian $\na^2 h$ of the entropy density \eqref{def.h} in the
variables $c_1,\ldots,c_N$. Then $H=(h_{ij})=\na^2 h\in\R^{N\times N}$ is given by
\begin{equation*}
  h_{ij} = \frac{1}{c_{N+1}} + \frac{\delta_{ij}}{c_i}, \quad
  i,j=1,\ldots,N,
\end{equation*}
where $\delta_{ij}$ denotes the Kronecker delta.
The matrix $H$ is symmetric and positive definite by Sylvester's criterion,
since all principle minors $\det H_k$ of $H$ are positive:
$$
  \det H_k = (c_1\cdots c_k c_{N+1})^{-1}\left(\sum_{i=1}^k c_i + c_{N+1}\right) > 0,
  \quad k=1,\ldots,N.
$$

\begin{lemma}[Properties of $B$]\label{lem.b}
%Let $c\in\R^{N+1}$ be a strictly positive vector.
The matrix $B=A_0^{-1}H^{-1}$ is symmetric and positive definite.
Furthermore, the elements of $B$ are bounded uniformly in $c_1,\ldots,c_{N+1}\in[0,1]$.
\end{lemma}

\begin{proof}
Using $d_{ij}=d_{ji}$ and $\sum_{i=1}^{N+1}c_i=1$, a calculation
shows that the elements $\beta_{ij}$ of $B^{-1}=HA_0$ equal
\begin{align*}
  \beta_{ii} &= d_{i,N+1}\left(1-\sum_{k=1,\,k\neq i}^N c_k\right)
  \left(\frac{1}{c_i}+\frac{1}{c_{N+1}}\right)
  + \sum_{k=1,\,k\neq i}^N\left(\frac{d_{k,N+1}}{c_{N+1}}+\frac{d_{ik}}{c_i}\right)c_k,
  \\
  \beta_{ij} &= \frac{d_{i,N+1}}{c_{N+1}}\left(1-\sum_{k=1,\,k\neq i}^N c_k\right)
  + \frac{d_{j,N+1}}{c_{N+1}}\left(1-\sum_{k=1,\,k\neq j}^N c_k\right)
  + \sum_{k=1,\,k\neq i,j}^N d_{k,N+1}\frac{c_k}{c_{N+1}} - d_{ij},
\end{align*}
where $i,j=1,\ldots,N$ and $i\neq j$.
Hence, $B^{-1}$ is symmetric. We have proved
above that $H^{-1}$ is symmetric and positive definite. 
According to Theorem \ref{thm.hpd}, the number of positive eigenvalues of 
$A_0=H^{-1}B^{-1}$ equals that for $B^{-1}$. However, by
\eqref{sigma.a0}, $A_0$ has only positive eigenvalues. Therefore, also $B^{-1}$ 
has only positive eigenvalues. This shows that $B^{-1}$ and consequently
$B$ are symmetric and positive definite.

It remains to show the uniform boundedness of $B$. The inverse $H^{-1}=(\eta_{ij})$
can be computed explicitly:
$$
  \eta_{ij} = \left\{\begin{array}{cl}
  (1-c_i)c_i &\quad\mbox{if }i=j=1,\ldots,N, \\
  -c_i c_j &\quad\mbox{if }i\neq j,\ i,j=1,\ldots,N.
  \end{array}\right.
$$
Denoting the elements of $A_0^{-1}$ by $\alpha_{ij}$, the elements
$b_{ij}$ of $B$ equal
\begin{align*}
  b_{ii} &=   \alpha_{ii}(1-c_i)c_i - \sum_{k=1,\,k\neq i}^N \alpha_{ik}c_i c_k, 
  \quad i=1,\ldots,N, \\[-4mm]
  b_{ij} &= -\alpha_{ii}c_ic_j + \alpha_{ij}(1-c_j)c_j
  - \sum_{k=1,\,k\neq i,j}^N\alpha_{ik}c_j c_k, \quad i\neq j,\ i,j=1,\ldots,N.
\end{align*}
By Lemma \ref{lem.a0}, the elements $\alpha_{ij}$ are uniformly bounded.
Then, since $c_i\in[0,1]$, the uniform bound for $b_{ij}$ follows.
\end{proof}

%%%%%%%%%%%%%%%%%%%%%%%%%%%%%%%%%%%%%%%%%%%%%%%%%%%%%%%%%%%%%%%%%%%%%%%%%%%%%%%%  

\section{Proof of Theorem \ref{thm.ex}}\label{sec.ex}

The analysis in Section \ref{sec.matrix} shows that 
the Maxwell-Stefan system can be reduced to the problem
\begin{align}
  & \pa_t c' - \diver(B(w)\na w) = r'(c)\quad\mbox{in }\Omega,\ t>0, \label{req1} \\
  & \na w_i\cdot\nu = 0\quad\mbox{on }\pa\Omega, \quad w_i(\cdot,0)=w_i^0 \quad
  \mbox{in }\Omega, \quad i=1,\ldots,N, \label{req2}
\end{align}
where $B=B(w)$ is symmetric and positive definite for $c_1,\ldots,c_{N+1}>0$
and $c_i=c_i(w)$ is given by \eqref{inv.w}. Furthermore, any solution $c'=c'(w)$
to this problem defines formally a solution to the original problem 
\eqref{eq1}-\eqref{bic} by setting $c_{N+1}=1-\sum_{i=1}^N c_i$.
We assume that there exists $0<\eta<1$ such that $c_i^0\ge\eta$ for $i=1,\ldots,N$ and
$c_{N+1}^0=1-\sum_{i=1}^N c_i^0\ge \eta$. Then 
$w_i^0=\log (c_i^0/c_{N+1}^0)$ satisfies $w_i^0\in L^\infty(\Omega)$, $i=1,\ldots,N$.

{\em Step 1: Existence of an approximate system.}
Let $T>0$, $m\in\N$ and set $\tau=T/m$, $t_k=\tau k$ for $k=0,\ldots,m$. 
We prove the existence of weak solutions to the approximate system
\begin{align}
  \frac{1}{\tau}\int_\Omega & \big(c'(w^k)-c'(w^{k-1})\big)\cdot vdx
  + \int_\Omega \na v: B(w^k)\na w^k dx \nonumber \\
  &{}+ \eps\int_\Omega(\Delta w^k\cdot\Delta v + w^k\cdot v)dx
  = \int_\Omega r'(c(w^k))\cdot v dx, \quad v\in \V^N, \label{w1}
\end{align}
where $w^k$ approximates $w(\cdot,t_k)$ and $\eps>0$.
The notation ``:'' signifies summation over both matrix indices; in particular,
$$
  \int_\Omega \na v: B(w^k)\na w^k dx
  = \sum_{i,j=1}^N\int_\Omega b_{ij}(w^k)\na v_i\cdot\na w_j^k dx,
$$
and we recall that $\V=\{u\in H^2(\Omega):\na u\cdot\nu=0$ on $\pa\Omega\}$.
The implicit Euler discretization of the time derivative makes the system elliptic
which avoids problems related to the regularity in time. The additional $\eps$-term
guarantees the coercivity of the elliptic system.

\begin{lemma}\label{lem.ex.approx}
Let the assumptions of Theorem \ref{thm.ex} hold and let 
$w^{k-1}\in L^\infty(\Omega)^N$. Then there exists a weak
solution $w^k\in \V^N$ to \eqref{w1}.
\end{lemma}

\begin{proof}
The idea of the proof is to apply the Leray-Schauder fixed-point theorem.
Let $\bar w\in L^\infty(\Omega)^N$ and $\sigma\in[0,1]$. 
We wish to find $w\in\V^N$ such that
\begin{equation}\label{LM}
  a(w,v) = F(v) \quad\mbox{for all }v\in \V^N,
\end{equation}
where
\begin{align*}
  a(w,v) &= \int_\Omega \na v: B(\bar w)\na w dx
  + \eps\int_\Omega(\Delta w\cdot\Delta v + w\cdot v)dx, \\
  F(v) &= -\frac{\sigma}{\tau}\int_\Omega\big(c'(\bar w)-c'(w^{k-1})\big)\cdot v dx
  + \sigma\int_\Omega r'(c(\bar w))\cdot v dx.
\end{align*}
Since $B(\bar w)$ is positive definite, by Lemma \ref{lem.b}, 
the bilinear form $a$ is coercive,
$$
  a(w,w) = \int_\Omega\sum_{i,j=1}^N  b_{ij}(\bar w)\na w_i\cdot\na w_j dx
  + \eps\int_\Omega\big(|\Delta w|^2+|w|^2\big)dx
  \ge C\eps\|w\|_{H^2(\Omega)^N}^2,
$$
where $C>0$ is a constant. The inequality follows from elliptic regularity,
using the assumption $\pa\Omega\in C^{1,1}$.
By Lemma \ref{lem.b} again, the elements of $B(\bar w)$ are bounded uniformly in $c$,
and thus, $a$ is continuous in $\V^N\times \V^N$.
Using $0<c_i(\bar w)$, $c_i(w^{k-1})<1$ and the continuity of $r_i$, we can show that
$F$ is bounded in $\V^N$. Then the Lax-Milgram lemma provides the
existence of a unique solution $w\in \V^N$ to \eqref{LM}.
Since the space dimension is assumed to be at most three, the embedding
$H^2(\Omega)\hookrightarrow L^\infty(\Omega)$ is continuous (and compact) 
such that $w\in L^\infty(\Omega)^N$.
This shows that the fixed-point operator $S:L^\infty(\Omega)^N\times[0,1]
\to L^\infty(\Omega)^N$, $S(\bar w,\sigma)=w$, is well-defined.
By construction, $S(\bar w,0)=0$ for all $\bar w\in
L^\infty(\Omega)^N$. Standard arguments show that $S$ is continuous and compact.
It remains to prove a uniform bound for all fixed points of $S(\cdot,\sigma)$
in $L^\infty(\Omega)^N$. 

Let $w\in L^\infty(\Omega)^N$ be such a fixed point. Then $w$ solves \eqref{LM}
with $\bar w$ replaced by $w$. Taking the test function $v=w\in\V^N$, it follows that
\begin{align}
  \frac{\sigma}{\tau}\int_\Omega & \big(c'(w)-c'(w^{k-1})\big)\cdot w dx
  + \int_\Omega\big(\na w: B(w)\na w + \eps(|\Delta w|^2 + |w|^2)\big)dx 
  \nonumber \\
  &{}= \sigma\int_\Omega r'(c(w))\cdot w dx. \label{aux1}
\end{align}

In order to estimate the first term on the left-hand side, we consider the
entropy density $h$, defined in \eqref{def.h}. Its Hessian is 
positive definite if $c_1,\ldots,c_{N+1}>0$ and hence, $h$ is convex, i.e.
$$
  h(c) - h(\hat c) \le \na h(c)\cdot(c-\hat c)
  \quad\mbox{for all }c,\hat c\in\R^N\mbox{ with }
  0<c_i,\hat c_i,\sum_{j=1}^Nc_j,\sum_{j=1}^N\hat c_j<1.
$$
Using $w=\na h(c')$, we find that
$$
  \frac{\sigma}{\tau}\int_\Omega\big(c'(w)-c'(w^{k-1})\big)\cdot w dx
  \ge \frac{\sigma}{\tau}\int_\Omega\big(h(c'(w))-h(c'(w^{k-1}))\big)dx.
$$
By Lemma \ref{lem.b}, $B$ is positive definite:
$$
   \int_\Omega \na w: B(w)\na w dx \ge 0.
$$
Finally, using the assumptions $\sum_{i=1}^{N}r_i(c)=-r_{N+1}(c)$ and 
$\sum_{i=1}^{N+1}r_i(c)\log c_i\le 0$, 
\begin{align*}
  \int_\Omega r'(c(w))\cdot w dx 
  &= \int_\Omega\left(\sum_{i=1}^N r_i(c(w))(\log c_i(w)-\log c_{N+1}(w))\right)dx \\
  &= \int_\Omega\sum_{i=1}^{N+1}r_i(c(w))\log c_i(w) dx \le 0.
\end{align*}
Therefore, \eqref{aux1} becomes
$$
  \sigma\int_\Omega h(c'(w))dx + \eps\tau\int_\Omega\big(|\Delta w|^2+|w|^2\big)dx
  \le \sigma\int_\Omega h(c'(w^{k-1}))dx.
$$
This yields an $H^2$ bound uniform in $w$ and $\sigma$ (but depending on $\eps$
and $\tau$).
The embedding $H^2(\Omega)\hookrightarrow L^\infty(\Omega)$ implies the desired
uniform bound in $L^\infty(\Omega)$, and the Leray-Schauder theorem gives a solution
to \eqref{w1}.
\end{proof}

Note that we obtain the uniform bounds also under the weaker condition
\eqref{ass.r2}. In this case, \eqref{aux1} can be estimated as
\begin{equation}\label{gen.r}
  \sigma\int_\Omega h(c'(w))dx + \eps\tau\int_\Omega\big(|\Delta w|^2+|w|^2\big)dx
  \le \sigma\int_\Omega h(c'(w^{k-1}))dx + \sigma C_r \tau\mbox{meas}(\Omega).
\end{equation}

{\em Step 2: Entropy dissipation.}
Since the diffusion matrix $B(w^k)$ defines a self-adjoint endomorphism, the
entropy-dissipation estimate
$$
  \int_\Omega \na w^k:B(w^k)\na w^k dx \ge \int_\Omega \lambda|\na w^k|^2 dx
$$
holds, where $\lambda$ is the smallest eigenvalue of $B(w^k)$. Unfortunately,
$\lambda$ depends on $c(w^k)$ and we do not have a 
positive lower bound independent of $w^k$.
However, we are able to prove an entropy-dissipation inequality in the variables
$\sqrt{c_i(w^k)}$ with a uniform positive lower bound. In the following,
we employ the notation $f(c^k)=(f(c_1^k),\ldots,f(c_{N+1}^k))^\top$ for
arbitrary functions $f$.

\begin{lemma}\label{lem.ed}
Let $w^k\in \V^N$ be a weak solution to \eqref{w1}. Then
$$
  \int_\Omega \na w^k:B(w^k)\na w^k dx \ge \frac{4}{\Delta}\int_\Omega 
  |\na\sqrt{c^k}|^2 dx,
$$
where $c^k=c(w^k)=(c_1(w^k),\ldots,c_{N+1}(w^k))^\top$ is defined in \eqref{inv.w}
and $c_{N+1}(w^k)=1-\sum_{i=1}^N c_i(w^k)$.
\end{lemma}

\begin{proof}
First, we claim that
$$
  \int_\Omega \na w^k:B(w^k)\na w^k dx
  = \int_\Omega \na\log c^k:(-\tilde A)^{-1}\na c^kdx,
$$
where $\tilde A$ is defined in Lemma \ref{lem.tildea}. 
To prove this identity we set $z'=(z_1,\ldots,z_N)^\top=B(w^k)\na w^k\in\R^{N\times d}$ 
and $z_{N+1}=-\sum_{i=1}^N z_i\in\R^d$. Then the definitions of $w^k$ and $z_{N+1}$
yield
\begin{align}
  \na w^k: B(w^k)\na w^k 
  &= \na w^k:z'
  = \sum_{i=1}^N\big(\na\log c_i^k-\na\log c_{N+1}^k\big)\cdot z_i \nonumber \\
  &= \sum_{i=1}^{N+1}\na\log c_i^k\cdot z_i = \na\log c^k:z, \label{z1}
\end{align}
where $z=(z',z_{N+1})^\top$. Using $\na w^k = H\na c'(w^k)$ and $B=A_0^{-1}H^{-1}$,
where $H=H(c'(w^k))$ is the Hessian of $h$, it follows that $z'=A_0^{-1}\na c'(w^k)$
or, equivalently, $\na c'(w^k)=A_0 z'$. A computation shows that for $i=1,\ldots,N$,
$$
  \na c_i^k = (A_0 z')_i = \sum_{j=1,\,j\neq i}^N(d_{ij}-d_{i,N+1})(z_i c_j^k-z_j c_i^k)
  + d_{i,N+1} z_i = (-Az)_i = (-\tilde Az)_i,
$$
since each column of $z$ is an element of $\mbox{im}(A)$. 
Because of $\tilde Az\in \text{im}(A)$, we have 
$(-\tilde Az)_{N+1} = -\sum_{i=1}^N(-\tilde Az)_i=\na c_{N+1}^k$. 
We infer that $\na c^k=-\tilde A z$ and 
consequently, $z=(-\tilde A)^{-1}\na c^k$. Inserting this into \eqref{z1}
proves the claim.

We recall from the proof of Lemma \ref{lem.tildea} that the images of
$\tilde A=A|_{\text{im}(A)}$ and $\tilde A_S=A_S|_{\text{im}(A_S)}$ are given by
$\mbox{im}(A)=\{\mathbf{1}\}^\perp$ and 
$\mbox{im}(A_S)=\mbox{span}\{\sqrt{c^k}\}^\perp = \{C^{-1/2}x:x\in \mbox{im}(A)\}$,
where $C^{\pm 1/2}=\mbox{diag}((c^k_1)^{\pm 1/2},\ldots,(c^k_{N+1})^{\pm 1/2})
\in\R^{(N+1)\times(N+1)}$.
Then the definition $-A=C^{1/2}(-A_S)C^{-1/2}$ implies that 
$-\tilde A=C^{1/2}(-\tilde A_S)C^{-1/2}$
and hence, $(-\tilde A_S)^{-1}=C^{-1/2}(-\tilde A)^{-1}C^{1/2}$. 
We infer that
\begin{align*}
  \na\log c^k:(-\tilde A)^{-1}\na c^k  
  &= 4(\na\sqrt{c^k}): C^{-1/2}(-\tilde A)^{-1}C^{1/2}\na\sqrt{c^k} \\
  &= 4\na\sqrt{c^k}:(-\tilde A_S)^{-1}\na\sqrt{c^k}
  \ge \frac{4}{\Delta}|\na\sqrt{c^k}|^2.
\end{align*}
The inequality follows from Lemma \ref{lem.tildea} since $(-\tilde A_S)^{-1}$ is a
self-adjoint endomorphism whose smallest eigenvalue is larger than $1/\Delta$.
\end{proof}

{\em Step 3: A priori estimates.}
Next, we derive some estimates uniform in $\tau$, $\eps$, and $\eta$ by means of the
entropy-dissipation inequality. 
%The notation $h(c)$ for the entropy density means that we interpret \eqref{def.h} 
%as a function of $c_1,\ldots,c_{N+1}$.
The following lemma is a consequence of \eqref{aux1}, the proof of
Lemma \ref{lem.ex.approx}, and Lemma \ref{lem.ed}.

\begin{lemma}[Discrete entropy inequality]\label{lem.ent}
Let $w^k\in \V^N$ be a weak solution to \eqref{w1}. Then for $k\ge 1$,
$$
  {\mathcal H}[c^k] + \frac{4\tau}{\Delta}\int_\Omega|\na\sqrt{c^k}|^2 dx
  + \eps\tau\int_\Omega\big(|\Delta w^k|^2 + |w^k|^2\big)dx \le {\mathcal H}[c^{k-1}],
$$
where $c^k=c(w^k)$ and ${\mathcal H}[c^k]=\int_\Omega h(c^k)dx$. Solving this
estimate recursively, it follows that
$$
  {\mathcal H}[c^k] + \frac{4\tau}{\Delta}\sum_{j=1}^k\int_\Omega|\na\sqrt{c^j}|^2 dx
  + \eps\tau\sum_{j=1}^k\int_\Omega\big(|\Delta w^j|^2 + |w^j|^2\big)dx 
  \le {\mathcal H}[c^{0}].
$$
\end{lemma}

Let $w^k\in \V^N$ be a weak solution to \eqref{w1} and set $c^k=c(w^k)$.
We define the piecewise-constant-in-time functions $w^{(\tau)}(x,t)=w^k(x)$
and $c^{(\tau)}(x,t)=(c^k_1,\ldots,c_N^k)^\top(x)$ for $x\in\Omega$, 
$t\in((k-1)\tau,k\tau]$, $k=1,\ldots,m$, $c^{(\tau)}(\cdot,0)=(c_1^0,\ldots,c_N^0)^\top$,
and we introduce the discrete time derivative $D_\tau c^{(\tau)} 
= (c^{(\tau)}-\sigma_\tau c^{(\tau)})/\tau$ with the shift operator
$(\sigma_\tau c^{(\tau)})(x,t)=c^{(\tau)}(x,t-\tau)$ for $x\in\Omega$, $t\in(\tau,T]$,
$(\sigma_\tau c^{(\tau)})(x,t)=c^0(x)$ for $x\in\Omega$, $t\in(0,\tau]$.
The functions $(c^{(\tau)},w^{(\tau)})$ solve the following equation in the
distributional sense:
\begin{equation}\label{c1}
  D_\tau c^{(\tau)} - \diver(A_0^{-1}(c^{(\tau)})\na c^{(\tau)})
  + \eps(\Delta^2 w^{(\tau)} + w^{(\tau)}) = r'(c^{(\tau)}),
  \quad t>0.
\end{equation}
Lemma \ref{lem.ent} implies the following a priori estimates.

\begin{lemma}\label{lem.est}
There exists a constant $C>0$ independent of $\eps$, $\tau$, and $\eta$ such that
\begin{align}
%  \|c^{(\tau)}\log c^{(\tau)}\|_{L^\infty(0,T;L^1(\Omega))}
  \|\sqrt{c^{(\tau)}}\|_{L^2(0,T;H^1(\Omega))}
  + \sqrt{\eps}\|w^{(\tau)}\|_{L^2(0,T;H^2(\Omega))}
  &\le C, \label{est1} \\
  \|c^{(\tau)}\|_{L^2(0,T;H^1(\Omega))}
  + \|D_\tau c^{(\tau)}\|_{L^2(0,T;\V')} &\le C. \label{est2}
\end{align}
\end{lemma}

In the following, $C>0$ denotes a generic constant independent of
$\eps$, $\tau$, and $\eta$.

\begin{proof}
Estimate \eqref{est1} is an immediate consequence of the entropy inequality
of Lemma \ref{lem.ent} and the $L^\infty$-bound for $c^{(\tau)}$. 
To prove \eqref{est2}, we employ the H\"older inequality:
\begin{align*}
  \|\na c_i^{(\tau)}\|_{L^2(0,T;L^2(\Omega))}^2
  &= 4\int_0^T\big\|\sqrt{c_i^{(\tau)}}\na\sqrt{c_i^{(\tau)}}
  \big\|_{L^2(\Omega)}^2 dt \\
  &\le 4\int_0^T\big\|\sqrt{c_i^{(\tau)}}\big\|_{L^\infty(\Omega)}^2
  \big\|\na\sqrt{c_i^{(\tau)}}\big\|_{L^2(\Omega)}^2 dt \\
  &\le 4\|c_i^{(\tau)}\|_{L^\infty(0,T;L^\infty(\Omega))}
  \big\|\na\sqrt{c_i^{(\tau)}}\big\|_{L^2(0,T;L^2(\Omega))}^2 \le C,
\end{align*}
using \eqref{est1} and the fact that $0<c_i^{(\tau)}<1$, $i=1,\ldots,N$.
Here and in the following, $C>0$ denotes a generic constant independent of
$\eps$, $\tau$, and $\eta$.
By \eqref{c1} and $L^2(\Omega)\hookrightarrow (H^1(\Omega))'$,
\begin{align*}
  \|D_\tau c^{(\tau)} \|_{L^2(0,T;\V')}
  &\le C\|A_0^{-1}\na c^{(\tau)}\|_{L^2(0,T;L^2(\Omega))} \\
  &\phantom{xx}{}+ \eps C\|w^{(\tau)}\|_{L^2(0,T;H^{2}(\Omega))}
  + C\|r'(c^{(\tau)})\|_{L^2(0,T;L^{2}(\Omega))} \\
  &\le C\|A_0^{-1}\|_{L^\infty(0,T;L^\infty(\Omega))}
  \|\na c^{(\tau)}\|_{L^2(0,T;L^{2}(\Omega))} \\
  &\phantom{xx}{}+ \eps C\|w^{(\tau)}\|_{L^2(0,T;H^{2}(\Omega))}
  + C\|r'(c^{(\tau)})\|_{L^2(0,T;L^{2}(\Omega))}.
\end{align*}
The proof of Lemma \ref{lem.a0} shows that the elements of $A_0^{-1}$ are
bounded by a constant which depends only on $N$ and $(D_{ij})$. 
Since $0<c^{(\tau)}<1$ and $r'$ is continuous, $(r'(c^{(\tau)}))$ is bounded
in $L^2(0,T;L^{2}(\Omega))$. Therefore, in view of \eqref{est1} and the bound
on $c^{(\tau)}$ in $L^2(0,T;H^1(\Omega))$,
$$
  \|D_\tau c^{(\tau)}\|_{L^2(0,T;\V')} \le C,
$$
finishing the proof.
\end{proof}

{\em Step 4: Limits $\eps\to 0$ and $\tau\to 0$.} We apply the compactness result
of \cite[Theorem 1]{DrJu12} to the family $(c^{(\tau)})$. Since the embedding
$H^{1}(\Omega)\hookrightarrow L^p(\Omega)$ is compact for $1<p<6$,
\eqref{est2} implies the existence of a subsequence, which is not relabeled,
such that, as $(\eps,\tau)\to 0$,
$$
  c^{(\tau)}\to c'=(c_1,\ldots,c_N) 
  \quad\mbox{strongly in }L^2(0,T;L^p(\Omega)), \quad 1<p<6.
$$
As a consequence, $c_i\ge 0$, $\sum_{i=1}^N c_i\le 1$, and 
$c_{N+1}=1-\sum_{i=1}^N c_i\ge 0$.
Because of the uniform $L^\infty$-bounds for $c_i^{(\tau)}$, this convergence
holds even in the space $L^q(\Omega\times(0,T))$ for all $1\le q<\infty$.
Furthermore, by \eqref{est1}-\eqref{est2}, up to subsequences,
\begin{align*}
  \na c^{(\tau)} \rightharpoonup \na c' &\quad\mbox{weakly in }
  L^{2}(0,T;L^{2}(\Omega)), \\
  D_\tau c^{(\tau)} \rightharpoonup \pa_t c' &\quad\mbox{weakly in }
  L^2(0,T;\V'), \\
  \eps w^{(\tau)} \to 0 &\quad\mbox{strongly in }L^2(0,T;H^2(\Omega)).
\end{align*}
Since the elements of $A_0^{-1}$ are bounded and $0<c^{(\tau)}<1$,
\begin{align*}
  A_0(c^{(\tau)})^{-1} \to A_0(c')^{-1} &\quad\mbox{strongly in }L^q(0,T;L^q(\Omega))
  \mbox{ for all }1\le q<\infty, \\
  r'(c^{(\tau)}) \to r'(c) &\quad\mbox{strongly in }L^2(0,T;L^2(\Omega)),
\end{align*}
setting $c=(c_1,\ldots,c_{N+1})^\top$.
The above convergence results are sufficient to pass to the limit
$(\eps,\tau)\to 0$ in the weak formulation of \eqref{c1}, showing that
$c$ satisfies
$$
  \pa_t c' - \diver(A_0(c')^{-1}\na c') = r'(c)\quad\mbox{in }
  L^{2}(0,T;\V').
$$
This proves the existence of a weak solution to \eqref{eq.A0} and \eqref{bic}
with initial data satisfying $c_i^0\ge\eta>0$ and $\sum_{i=1}^N c_i^0\le 1-\eta$.
In view of the uniform bounds and the finiteness of the initial entropy,
we can perform the limit $\eta\to 0$ to obtain the existence result for
general initial data with $c_i^0\ge 0$ and $\sum_{i=1}^N c_i^0\le 1$.
This proves Theorem \ref{thm.ex}.

%%%%%%%%%%%%%%%%%%%%%%%%%%%%%%%%%%%%%%%%%%%%%%%%%%%%%%%%%%%%%%%%%%%%%%%%%%%%%%%%  

\section{Proof of Theorem \ref{thm.long}}\label{sec.long}

First, we prove that, if the production rates vanish, the $L^1$ norms of
the semi-discrete molar concentrations are bounded. We assume that there
exists $0<\eta<1$ such that $c_i^0\ge\eta$ for $i=1,\ldots,N+1$.

\begin{lemma}[Bounded $L^1$ norms]\label{lem.L1}
Let $r=0$. Then there exists a constant $\gamma_0>0$, only depending on 
$c^0$, such that for all $0<\gamma\le\min\{1,\gamma_0\}$ 
and sufficiently small $\eps>0$,
depending on $\gamma$, the semi-discrete concentrations $c^k=c(w^k)$,
where $w^k\in \V^N$ solves \eqref{w1}, satisfy
\begin{align}
  & (1-\gamma)\|c_i^0\|_{L^1(\Omega)} \le \|c_i^k\|_{L^1(\Omega)} 
  \le (1+\gamma)\|c_i^0\|_{L^1(\Omega)}, \quad i=1,\ldots,N, \quad
  k\in\N, \label{mass1} \\
  & \|c_{N+1}^0\|_{L^1(\Omega)} - \gamma\sum_{i=1}^N\|c_i^0\|_{L^1(\Omega)} 
  \le \|c_{N+1}^k\|_{L^1(\Omega)}
  \le \|c_{N+1}^0\|_{L^1(\Omega)} + \gamma\sum_{i=1}^N\|c_i^0\|_{L^1(\Omega)}.
  \label{mass2}
\end{align}
Furthermore, $\|c_{N+1}^k\|_{L^1(\Omega)}\ge \frac12\|c_{N+1}^0\|_{L^1(\Omega)}>0$.
\end{lemma}

\begin{proof}
We recall that $\tau=T/m$ for $T>0$ and $m\in\N$. Let $k\in\{1,\ldots,m\}$.
Using the test function $v=e_i$ in \eqref{w1}, where $e_i$ is the $i$-th unit
vector of $\R^N$, we find that
$$
  \int_\Omega c^k_i dx = \int_\Omega c^{k-1}_i dx - \eps\tau\int_\Omega w_i^k dx,
  \quad i=1,\ldots,N,
$$
where we abbreviated $c^k_i=c_i(w^k)$. Solving these recursive equations, we obtain
\begin{equation}\label{aux11}
  \int_\Omega c^k_i dx = \int_\Omega c^0_i dx 
  - \eps\tau\sum_{j=1}^k\int_\Omega w_i^j dx.
\end{equation}
Because of the $\eps$-terms, we do not have discrete mass conservation 
but we will derive uniform $L^1$-bounds. 
The entropy inequality in Lemma \ref{lem.ent} shows that
\begin{align}
  {\mathcal H}[c^k] + \eps\tau\int_\Omega\big((\Delta w_i^k)^2 + (w_i^k)^2\big)dx
  &\le {\mathcal H}[c^k] 
  + \eps\tau\int_\Omega\big(|\Delta w^k|^2 + |w^k|^2\big)dx \nonumber \\
  &\le {\mathcal H}[c^{k-1}], \quad k\ge 1. \label{ent.w}
\end{align}
Solving this inequality recursively, we infer from 
${\mathcal H}[c^k]\ge -\text{meas}(\Omega)(N+1)$ that
$$
  \eps\tau \sum_{j=1}^k\|w_i^j\|_{L^2(\Omega)}^2 
  \le {\mathcal H}[c^0] - {\mathcal H}[c^k] \le {\mathcal H}[c^0] 
  + \text{meas}(\Omega)(N+1).
$$
Consequently, using $k\tau\le T$,
\begin{align*}
  \eps\tau \sum_{j=1}^k\int_\Omega|w_i^j|dx
  &\le \eps\tau C\sum_{j=1}^k\|w_i^j\|_{L^2(\Omega)}
  \le \eps\tau C\sqrt{k}\left(\sum_{j=1}^k\|w_i^j\|_{L^2(\Omega)}^2\right)^{1/2} \\
  &\le C\sqrt{\eps\tau k ({\mathcal H}[c^0]+\text{meas}(\Omega)(N+1))} \\
  &\le C\sqrt{\eps T({\mathcal H}[c^0]+\text{meas}(\Omega)(N+1))}.
\end{align*}
Let $\gamma>0$ and $0<\eps<1$ satisfy
\begin{align}
  & 0 < \gamma \le 
  \min\left\{1,\gamma_0=\left(2\sum_{i=1}^N \|c_i^0\|_{L^1(\Omega)}\right)^{-1}
  \|c_{N+1}^0\|_{L^1(\Omega)}\right\}, \label{gamma1} \\
  & 0 < \sqrt{\eps} 
  \le \frac{\gamma\min_{\ell=1,\ldots,N}\|c_\ell^0\|_{L^1(\Omega)}}{C
  \sqrt{T({\mathcal H}[c^0]+\text{meas}(\Omega)(N+1))}}. \label{gamma2}
\end{align}
Then, in view of \eqref{aux11},
$$
  (1-\gamma)\|c_i^0\|_{L^1(\Omega)} \le \|c_i^k\|_{L^1(\Omega)} 
  = \|c_i^0\|_{L^1(\Omega)} - \eps\tau\sum_{j=1}^k\int_\Omega w_i^j dx
  \le (1+\gamma)\|c_i^0\|_{L^1(\Omega)}.
$$
These relations hold for all $i=1,\ldots,N$. For $i=N+1$, we estimate
(using \eqref{mass1})
\begin{align*}
  \int_\Omega c_{N+1}^k dx
  &= \int_\Omega\left(1-\sum_{i=1}^N c_i^k\right)dx
  \ge \int_\Omega\left(1-(1+\gamma)\sum_{i=1}^N c_i^0\right)dx \\
  &= \int_\Omega c_{N+1}^0 dx - \gamma\sum_{i=1}^N\int_\Omega c_i^0 dx 
  \ge \frac12\|c_{N+1}^0\|_{L^1(\Omega)} > 0,
\end{align*}
by definition of $\gamma_0$. A similar computation yields
$$
  \int_\Omega c_{N+1}^k dx
  \le \int_\Omega\left(1-(1-\gamma)\sum_{i=1}^N c_i^0\right)dx
  = \|c_{N+1}^0\|_{L^1(\Omega)} + \gamma\sum_{i=1}^N\|c_i^0\|_{L^1(\Omega)}.
$$
This proves the lemma.
\end{proof}

For the proof of Theorem \ref{thm.long}, we introduce the following notations: 
$c^k=(c_1^k,\ldots,c_{N+1}^k)^\top$, 
$\bar c^k=(\bar c_1^k,\ldots,\bar c_{N+1}^k)^\top$, where
$\bar c_i^k=\mbox{meas}(\Omega)^{-1}\int_\Omega c_i^k dx$ for $i=1,\ldots,N+1$,
$k\ge 0$. Furthermore, we set
$w^k=(w_1^k,\ldots,w_N^k)^\top$, and $\bar w^k=(\bar w_1^k,\ldots,\bar w_N^k)^\top$,
where $\bar w_i^k=\log(\bar c_i^k/\bar c_{N+1}^k)$ for $i=1,\ldots,N$.
We recall the definition of the relative entropy
$$
  {\mathcal H}^*[c^k] 
  = \sum_{i=1}^{N+1}\int_\Omega c_i^k\log\frac{c_i^k}{\bar c_i^0}dx.
$$
Employing the test function $w^k-\bar w^k$ in \eqref{w1}, we obtain
\begin{align*}
  \frac{1}{\tau}\int_\Omega(c'(w^k)-c'(w^{k-1}))\cdot(w^k-\bar w^k)dx
  &+ \int_\Omega \na w^k:B(w^k)\na w^k dx \\
  &+ \eps\int_\Omega(|\Delta w^k|^2 + w^k\cdot(w^k-\bar w^k))dx = 0.
\end{align*}
We estimate the integrals term by term.

Using the definition $c_{N+1}^k=1-\sum_{i=1}^N c_i^k$, a computation shows that
$$
  (c'(w^k)-c'(w^{k-1}))\cdot w^k = (c^k-c^{k-1})\cdot\log c^k.
$$
Therefore, we find that
\begin{align*}
  \int_\Omega(c'(w^k)-c'(w^{k-1}))\cdot(w^k-\bar w^k)dx
  &= \int_\Omega(c^k-c^{k-1})\cdot\log\frac{c^k}{\bar c^k}dx \\
  &= \int_\Omega(c^k-c^{k-1})\cdot\log\frac{c^k}{\bar c^0}dx
  + \int_\Omega(c^k-c^{k-1})\cdot\log\frac{\bar c^0}{\bar c^k}dx.
\end{align*}
The first integral on the right-hand side can be estimated by
employing the convexity of $h(c)$ as a function of $c_1,\ldots,c_{N+1}$, which 
implies that
\begin{align*}
  h(c(w^k))-h(c(w^{k+1})) &\le \na h(c(w^k))\cdot(c(w^k)-c(w^{k-1})) \\
  &= \log(c^k)\cdot(c^k-c^{k-1}).
\end{align*}
Thus, because of $\sum_{i=1}^{N+1}c_i^k=1$ and the definition of ${\mathcal H}^*[c^k]$,
$$
  \int_\Omega(c^k-c^{k-1})\cdot\log\frac{c^k}{\bar c^0}dx
  \ge {\mathcal H}^*[c^k] - {\mathcal H}^*[c^{k-1}].
$$
For the second integral, we employ the bounds \eqref{mass1}-\eqref{mass2}
as well as $\gamma<1$ and $\eps>0$ sufficiently small, which yields
$$
  \frac{1}{1+\gamma} \le \frac{\bar c_i^0}{\bar c_i^k} \le \frac{1}{1-\gamma}, \quad
  i=1,\ldots,N, \quad
  \frac{1}{1+\bar\gamma} \le \frac{\bar c_{N+1}^0}{\bar c_{N+1}^k} 
  \le \frac{1}{1-\bar\gamma},
$$
where $\bar\gamma=\gamma(1/\bar c_{N+1}^0-1)>0$.  
Here, we have used again that $\sum_{i=1}^{N+1}\bar c_i^0 = 1$.
Then, with $C_1=\mbox{meas}(\Omega)$,
\begin{align*}
  \int_\Omega(c^k-c^{k-1})\cdot\log\frac{\bar c^0}{\bar c^k}dx
  &\ge -\int_\Omega \sum_{i=1}^{N}c^k_i\log(1+\gamma)dx
  - \int_\Omega c_{N+1}^k\log(1+\bar\gamma)dx \\
  &\phantom{xx}{}+ \int_\Omega \sum_{i=1}^{N}c^{k-1}_i\log(1-\gamma)dx 
  + \int_\Omega c_{N+1}^{k-1}\log(1-\bar\gamma)dx \\
  &\ge -C_1\log\frac{(1+\gamma)(1+\bar\gamma)}{(1-\gamma)(1-\bar\gamma)}.
\end{align*}

We have already proved that
$$
  \int_\Omega \na w^k:B(w^k)\na w^k dx \ge \frac{4}{\Delta}\int_\Omega
  |\na\sqrt{c^k}|^2 dx.
$$
Applying Young's inequality to the $\eps$-term, it follows that
$$
  {\mathcal H}^*[c^k] - {\mathcal H}^*[c^{k-1}]
  + \frac{4\tau}{\Delta}\int_\Omega|\na\sqrt{c^k}|^2 dx 
  \le \frac{\eps\tau}{2}\int_\Omega |\bar w^k|^2 dx 
  + C_1\log\frac{(1+\gamma)(1+\bar\gamma)}{(1-\gamma)(1-\bar\gamma)}.
$$
The logarithmic Sobolev inequality \cite{Gro75} as well as the bounds
\eqref{mass1}-\eqref{mass2} show that
\begin{align*}
  {\mathcal H}^*[c^k] 
  &= \sum_{i=1}^{N+1}\int_\Omega c_i^k\log\frac{c_i^k}{\bar c_i^k}dx
  + \sum_{i=1}^{N+1}\int_\Omega c_i^k\log\frac{\bar c_i^k}{\bar c_i^0}dx \\
  &\le C(\Omega)\sum_{i=1}^{N+1}\int_\Omega|\na\sqrt{c^k_i}|^2 dx
  + \sum_{i=1}^{N}\int_\Omega c_i^k\log(1+\gamma)dx
  + \int_\Omega c_{N+1}^k\log(1+\bar\gamma)dx \\
  &\le C(\Omega)\int_\Omega|\na\sqrt{c^k}|^2 dx + C_1\log((1+\gamma)(1+\bar\gamma)),
\end{align*}
from which we infer that
$$
  (1 + C_2\tau){\mathcal H}^*[c^k]
  \le {\mathcal H}^*[c^{k-1}] + \frac{\eps\tau}{2}\int_\Omega |\bar w^k|^2 dx
  + C_\gamma,
$$
where $C_2=4/(C(\Omega)\Delta)$ and, for $\tau\le 1$,
$$
  C_\gamma = C_1\log\frac{(1+\gamma)(1+\bar\gamma)}{(1-\gamma)(1-\bar\gamma)}
  + \frac{4C_1}{C(\Omega)\Delta}\log((1+\gamma)(1+\bar\gamma)).
$$
We can estimate $\bar w^k$ by using the bounds for $\bar c^k$ of Lemma \ref{lem.L1}:
$$
  \int_\Omega |\bar w^k|^2 dx 
  \le \sum_{i=1}^N\int_\Omega\big(|\log\bar c_i^k| + |\log\bar c_{N+1}^k|\big)^2 dx 
  \le C_3,
$$
where $C_3>0$ depends on the $L^1$-norm of $c^0$ and $\gamma$. Hence
$$
  {\mathcal H}^*[c^k] \le (1 + C_2\tau)^{-1}
  {\mathcal H}^*[c^{k-1}] 
  + \frac{\eps\tau}{2}C_3(1 + C_2\tau)^{-1} + C_\gamma(1 + C_2\tau)^{-1}.
$$
Solving these recursive inequalities, we conclude that
$$
  {\mathcal H}^*[c^k] 
  \le (1 + C_2\tau)^{-k}
  {\mathcal H}^*[c^0] + \frac{\eps\tau}{2} C_3\sum_{j=1}^k(1 + C_2\tau)^{-j}
  + C_\gamma\sum_{j=1}^k(1 + C_2\tau)^{-j}.
$$
The sum contains the first terms of the geometric series:
$$
  \sum_{j=1}^k(1 + C_2\tau)^{-j}
  \le \frac{1}{1-(1 + C_2\tau)^{-1}} - 1
  = \frac{1}{C_2\tau},
$$
yielding
$$
  {\mathcal H}^*[c^{(\tau)}(\cdot,t)] 
  \le (1 + C_2\tau)^{-t/\tau}{\mathcal H}^*[c^0] + \frac{\eps C_3}{2C_2}
  +\frac{C_\gamma}{C_2\tau}, \quad t>0.
$$
Now, we choose sequences for $\eps$, $\tau$, and $\gamma$ such that
$\gamma\to 0$, $C_\gamma/\tau\to 0$, and \eqref{gamma2} is satisfied
(then also $\eps\to 0$).
This is possible since $C_\gamma\to 0$ as $\gamma\to 0$.
Then, because of $c^{(\tau)}_i\to c_i$ in $L^2(0,T;L^2(\Omega))$ for $i=1,\ldots,N+1$,
the limit $(\eps,\tau,\gamma)\to 0$ leads to
$$
  {\mathcal H}^*[c(\cdot,t)] \le e^{-C_2 t}{\mathcal H}^*[c^0], \quad t\ge 0.
$$
Moreover, we can pass to the limit $\eta\to 0$.
Finally, since $\int_\Omega c_i(\cdot,t)dx=\int_\Omega c_i^0 dx$ for
$i=1,\ldots,N+1$ (see Lemma \ref{lem.L1}), 
we can apply the Csisz\'ar-Kullback inequality \cite{UAMT00} to finish the proof.

%%%%%%%%%%%%%%%%%%%%%%%%%%%%%%%%%%%%%%%%%%%%%%%%%%%%%%%%%%%%%%%%%%%%%%%%%%%%%%%%  

\appendix

\section{Derivation of the Maxwell-Stefan relations}\label{sec.deriv}

The Maxwell-Stefan relations \eqref{eq2} for an ideal gas mixture of $N+1$ components
are derived by assuming that the thermodynamical driving force $d_i$ of the 
$i$-th component balances the friction force $f_i$. We suppose constant
temperature and pressure. Our derivation follows \cite{Bot11}.
For details on the modeling, we refer to the monographs \cite{Gio99,WeKr00}. 

The driving force $d_i$ is assumed to be proportional
to the gradient of the chemical potential $\mu_i$ \cite[Section 3.3]{WeKr00}:
\begin{equation}\label{di}
  d_i = \frac{c_i}{RT}\na\mu_i, \quad i=1,\ldots,N+1.
\end{equation}
Here, $R$ is the gas constant, $T$ the (constant) temperature, and 
$c_i=\rho_i/m_i$ is the molar concentration of the $i$-th species
with the mass density $\rho_i$ and the
molar mass $m_i$ of the $i$-th species. 
For more general expressions of $d_i$, we refer to \cite[Chapter 7]{Gio99}.
The chemical potential under isothermal, isobaric conditions is defined by
$\mu_i=\pa G/\pa c_i$, where $G$ is the Gibbs free energy. 
In an ideal gas, we have $\na\mu_i = RT\na\log c_i$ implying that $d_i=\na c_i$.

The mutual friction force between the $i$-th and the $j$-th component 
is supposed to be proportional to the relative velocity and the amount of 
molar mass such that
\begin{equation}\label{fi}
  f_i = -\sum_{j=1,\,j\neq i}^{N+1}\frac{c_i c_j(u_i-u_j)}{D_{ij}},
\end{equation}
where $D_{ij}$ are the binary Maxwell-Stefan diffusivities \cite[Formula (16)]{KrWe97}. 
By the Onsager reciprocal relation, the diffusion matrix $(D_{ij})$ 
is symmetric \cite{KrTa86}. 
Then, using \eqref{di}, \eqref{fi}, and the definition $J_k=c_ku_k$, 
the balance $d_i=f_i$ becomes
$$
  \na c_i = -\sum_{j=1,\,j\neq i}^{N+1}\frac{c_j J_i-c_i J_j}{D_{ij}}, \quad
  i=1,\ldots,N+1,
$$
which equals \eqref{eq2}. 

Another derivation of the Maxwell-Stefan equations \eqref{eq1} and \eqref{eq2} 
starts from the Boltzmann transport equation for an isothermal ideal gas 
mixture \cite{BGS11}.
The main assumptions are that a diffusion scaling is possible and that the scattering 
rates are independent of the microscopic velocities.
Then, in the formal limit of vanishing mean-free paths, \eqref{eq1} and \eqref{eq2}
are derived. The diffusivities $D_{ij}$ are determined by 
$$
  D_{ij} = \frac{T(m_i+m_j)}{m_i m_j S_{ij}},
$$
where $S_{ij}$ are the averaged scattering rates of the collision operator associated
to the components $i$ and $j$. Since $S_{ij}=S_{ji}$, the diffusion matrix
$(D_{ij})$ is symmetric. We also observe that $(D_{ij})$ does not depend on
the concentrations which is consistent with the assumption made in this paper.

%%%%%%%%%%%%%%%%%%%%%%%%%%%%%%%%%%%%%%%%%%%%%%%%%%%%%%%%%%%%%%%%%%%%%%%%%%%%%%%%

\section{Matrix analysis}\label{sec.app}

We recall some results on the eigenvalues of special matrices such as
symmetric, quasi-positive, or rank-one matrices. 
Although most of the results in this appendix are valid
for matrices with complex elements, we consider the real case only and
refer to the literature for the general situation \cite{HoJo85,Ser10,Thi03}.

A vector $x\in\R^n$ ($n\in\N$)
is called {\em positive} if all components are nonnegative and at
least one component is positive. 
It is called {\em strictly positive} if all components are positive \cite{Thi03}.
Let $A=(a_{ij})\in\R^{n\times n}$ be a square matrix. 
The unit matrix in  $\R^{n\times n}$ is denoted by $I_n$. 
Let $\sigma(A)$ denote the spectrum of $A$.
The spectral radius of $A$ is defined by $r(A)=\max\{|\lambda|:\lambda\in\sigma(A)\}$,
and the spectral bound of $A$ equals $s(A)=\max\{\Re(\lambda):\lambda\in\sigma(A)\}$.
An eigenvalue of $A$ is called {\em semisimple} if its algebraic and geometric
multiplicities coincide and {\em simple} if its algebraic multiplicity 
(and hence also its geometric multiplicity) equals one. The following theorem
is proved in \cite[Theorem 3.4]{Ser10}.

\begin{theorem}\label{thm.semi}
Let $A\in\R^{n\times n}$ and let $\lambda\in\sigma(A)$ be a real eigenvalue. 
Then $\lambda$ is semisimple if and only if
$$
  \R^n = \mbox{\rm im}(A-\lambda I_n)\oplus\mbox{\rm ker}(A-\lambda I_n).
$$
\end{theorem}

The matrix $A$ is called
{\em quasi-positive} if $A\neq 0$ and $a_{ij}\ge 0$ for all $i\neq j$
and {\em irreducible} if for any proper nonempty subset $M\subset\{1,\ldots,n\}$
there exist $i\in M$ and $j\not\in M$ such that $a_{ji}\neq 0$. 
If $n=1$, $A$ is called irreducible if $A\neq 0$. For quasi-positive
and irreducible matrices, the following result holds 
\cite[Theorem A.45, Remark A.46]{Thi03}.

\begin{theorem}[Perron-Frobenius]\label{thm.PF}
Let $A$ be a quasi-positive and irreducible matrix. 
Then its spectral bound $s(A)$ is a simple eigenvalue of
$A$ associated with a strictly positive eigenvector and
$s(A)>\Re(\lambda)$ for all $\lambda\in\sigma(A)$, $\lambda\neq s(A)$.
All eigenvalues of $A$ different from $s(A)$ have no positive eigenvector.
\end{theorem}

The spectrum of rank-one matrices can be determined explicitly
\cite[Section 3.8, Lemma 2]{Ser10}. Notice that
any rank-one matrix $A\in\R^{n\times m}$ can be written in the form
$A=x\otimes y$, where $x\in\R^n$, $y\in\R^m$. 

\begin{proposition}[Spectrum of rank-one matrix]\label{prop.xy}
Let $x$, $y\in\R^n$. Then $\sigma(x\otimes y)=\{0,\ldots,0,x\cdot y\}$, i.e., 
$x\cdot y$ is a simple eigenvalue.
\end{proposition}

We recall two results on eigenvalues of products and sums of symmetric matrices.

\begin{theorem}\label{thm.hpd}
Let $A\in\R^{n\times n}$ be symmetric and positive definite and let 
$B\in\R^{n\times n}$ be symmetric. Then the number
of positive eigenvalues of $AB$ equals that for $B$. 
\end{theorem}

For a proof, we refer to \cite[Prop.~6.1]{Ser10}.

\begin{theorem}[Weyl]\label{thm.weyl}
Let $A$, $B\in\R^{n\times n}$ be symmetric and let the eigenvalues $\lambda_i(A)$
of $A$ and $\lambda_i(B)$ of $B$ be arranged in increasing order. Then,
for $i=1,\ldots,n$,
$$
  \lambda_i(A) + \lambda_1(B) \le \lambda_i(A+B)\le \lambda_i(A) + \lambda_n(B).
$$
\end{theorem}

A proof is given in \cite[Theorem 4.3.1]{HoJo85}.

%%%%%%%%%%%%%%%%%%%%%%%%%%%%%%%%%%%%%%%%%%%%%%%%%%%%%%%%%%%%%%%%%%%%%%%%%%%%%%%%

\end{document}